\documentclass[a4paper,12pt]{article}

\usepackage{amsfonts}
\usepackage{amsmath}
\usepackage{amsthm}
\usepackage{hyperref}
\usepackage{mathtools}

\newtheorem{theorem}{Theorem}
\newtheorem{corollary}[theorem]{Corollary}
\newtheorem{lemma}[theorem]{Lemma}

\DeclareMathOperator{\sympeak}{sp}
\DeclareMathOperator{\symvalley}{sv}
\DeclareMathOperator{\hsympeak}{hsp}
\DeclareMathOperator{\dsymvalley}{dsv}

\DeclareMathOperator{\HSP}{HSP}
\DeclareMathOperator{\DSV}{DSV}
\DeclareMathOperator{\GHSP}{GHSP}
\DeclareMathOperator{\GSP}{GSP}
\DeclareMathOperator{\GDSV}{GDSV}
\DeclareMathOperator{\GSV}{GSV}

\begin{document}
\title{The heights of symmetric peaks and the depth of symmetric valleys over compositions of an integer}

\author{Walaa Asakly and Noor Kezil\\
Department of Mathematics, Braude college, Karmiel, Israel\\
Department of Mathematics, University of Haifa , Haifa, Israel\\
{\tt walaaa@braude.ac.il}\\
{\tt nkizil02@campus.haifa.ac.il}}
\date{\small }
\maketitle

\begin{abstract}
A composition $\pi=\pi_1\pi_2\cdots\pi_k$ of a positive integer $n$ is an ordered collection of one or more positive integers whose sum is $n$ . The number of summands, namely $k$, is called the number of parts of $\pi$.
In this paper, we introduce two statistics over compositions of an integer $n$ with exactly $k$ parts: heights of symmetric peaks and depths of symmetric valleys over all compositions of $n$. We derive an explicit formula for the generating functions of compositions of $n$ with exactly $k$ parts according to the number of symmetric peaks (valleys) and the total heights (depths) of peaks (valleys).

Additionally, we present a combinatorial proof for  the total heights (depths) of peaks (valleys) over all compositions of $n$ with exactly $k$ parts.
Furthermore, we investigate these statistics in the context of random words, where the letters are generated by geometric probabilities.
\medskip

\end{abstract}

\section*{Introduction}
This paper is divided into three sections. In the first and second sections, we focus on compositions of an integer $n$ as follows:
A {\em composition} $\pi$ of a positive integer $n$ with $k$ parts is a sequence $\pi_1\pi_2\cdots\pi_k$ of positive integers over the set
$[n]=\{1,2\ldots,n\}$ such that $\sum_{i=1}^{k}\pi_i=n$. The number of parts namely, $k$ is called the number of {\em parts} of $\pi$.
For instance, the compositions of $5$ are $5$, $41$, $14$, $32$, $23$, $113$, $131$, $311$, $221$, $212$, $122$, $1112$, $1211$, $1121$, $2111$ and $11111$.

Let $C_n$ $(C_{n,k}$) be the set of all compositions of $n$ (compositions of $n$ with exactly $k$ parts), with $|C_0|=|C_{0,0}|=1$, $|C_{0,k}|=0$ for $k\geq 1$, and $|C_{n,k}|=0$ for $n<0$. The number of compositions of $n$ with exactly $k$ parts is given by $\binom{n-1}{k-1}$, and the total number of compositions of $n$ is $2^{n-1}$, see \cite{Mac}.

If there exists $1\leq i\leq k-2$ such that $\pi_i<\pi_{i+1}$ and $\pi_i=\pi_{i+2}$ ( $\pi_i>\pi_{i+1}$ and $\pi_i=\pi_{i+2}$),
we say that the composition $\pi\in C_{n,k} $  contains a {\em symmetric peak (valley)} $\pi_i\pi_{i+1}\pi_{i+2}$.
Mansour, Moreno and Ram\'irez \cite{MMR} found the total number of symmetric peaks and symmetric valleys over all compositions of $n$.
We define the value $\pi_{i+1}-\pi_{i}$ ($\pi_{i}-\pi_{i+1}$) to be the {\em height of a symmetric peak (the depth of a symmetric valley)}.
The study of symmetric peaks and valleys interested researchers in various combinatorial structures. Asakly \cite{A} found the number of symmetric and non-symmetric peaks over words. More recently, Fl\'orez and  Ram\'irez \cite{FR} proposed the idea of symmetric and non symmetric peaks in Dyck paths. Following this, Elizalde \cite{E} generalized their results.

 In the first two sections we obtain the generating function for the number of compositions of an integer $n$ with exactly $k$ parts according to the statistics of the number of symmetric peaks and the sum of the heights of symmetric peaks. By using the theory of generating functions, we derive the formula of the sum of the heights of symmetric peaks.
A similar discussion leads us to derive the formula of the sum of the depths of symmetric valleys.  We provide combinatorial formulas to count the sum of heights and depths of symmetric peaks and symmetric valleys, respectively.

In the third section we study  geometrically distributed words. To implement this, we need the following definition: if $0< p \leq 1$, then a  discrete random variable $X$  is said to be {\em geometric}, if $P\{X=k\}=pq^{k-1}$ for all integers $k \geq 1$, where $p+q=1$. We say that a word
$\omega=\omega_1\omega_2\cdots \omega_m$ over the alphabet of positive integers is {\em geometrically distributed}, if the letters of $\omega$ are independent and identically distributed geometric random variables.
The study of geometrically distributed words has been a recent topic of interest in enumerative combinatorics; see for example, \cite{B, BK}, and the
references therein.
In the end of the section, we count the total number of symmetric peaks (valleys) and the sum of heights of symmetric peaks (depths of valleys) in a discrete geometrically distributed sample.

\vspace{0.5cm}

\section{Heights of symmetric peaks}
\subsection{The generating function for the number of compositions of $n$ according to sum of symmetric peak heights}
Let $\sympeak(\pi)$ and $\hsympeak(\pi)$ denote the number of symmetric peaks  and the sum of the heights of symmetric peaks in a composition $\pi$, respectively.
Denote  the generating function for the number of compositions of $n$ with exactly $k$ parts
according to the statistics  $\sympeak$ and $\hsympeak$ by $\HSP(x,y,q,h)$; that is,
$$\HSP(x,y,q,h)=\sum_{n,k\geq 0}\sum_{\pi \in C_{n,k}}x^n y^k q^{\sympeak(\pi)} h^{\hsympeak(\pi)}.$$

We derive an explicit formula for the generating function $\HSP(x,y,q,h)$. For that, we denote the generating function for the number of compositions $\pi=\pi_1\pi_2\cdots\pi_k\in C_{n,k}$ such that $\pi_j=a_j$ for all $j=1,2,\ldots,s$ and $s<k$ according to the statistics $\sympeak$ and $\hsympeak$ by
$$\HSP(x,y,q,h|a_1a_2\cdots a_s)=\sum_{n,k\geq 0}\sum_{\pi=a_1a_2\cdots a_s\pi_{s+1}\cdots\pi_k \in C_{n,k}}x^n y^k q^{\sympeak(\pi)} h^{\hsympeak(\pi)}.$$

\begin{theorem}\label{Th1}
 The generating function for the number of compositions of $n$ with exactly $k$ parts
according to the number of symmetric peaks and the sum of heights of symmetric peaks is given by
\begin{equation}\label{eq11}
\HSP(x,y,q,h)=\frac{1}{1-\sum_{a\geq1}\frac{x^ay}{1-y^2x^{2a+1}(\frac{qh}{1-hx}-\frac{1}{1-x})}}.
\end{equation}
\end{theorem}
\begin{proof}
By definition
\begin{equation}\label{eq1}
\HSP(x,y,q,h)=1+\sum_{a\geq1}\HSP(x,y,q,h|a).
\end{equation}
Moreover
\begin{align}\
\HSP(x,y,q,h|a)
&= x^ay+\sum_{j=1}^{a}\HSP(x,y,q,h|aj)+\sum_{j\geq a+1}\HSP(x,y,q,h|aj) \nonumber \\
&=x^ay+x^ay\sum_{j=1}^{a}\HSP(x,y,q,h|j)+\sum_{j\geq a+1}\HSP(x,y,q,h|aj).\label{eq2}
\end{align}
For all $b\geq a+1$
\begin{align*}
\HSP(x,y,q,h|ab)&=x^{a+b}y^2+\sum_{j=1, j\neq a}^{b-1}\HSP(x,y,q,h|abj)+\HSP(x,y,q,h|aba)\\
&+\HSP(x,y,q,h|abb)+\sum_{j\geq b+1}\HSP(x,y,q,h|abj)\\
&=x^{a+b}y^2+x^{a+b}y^2\sum_{j=1, j\neq a}^{b-1}\HSP(x,y,q,h|j)+qx^{a+b}y^2h^{b-a}\HSP(x,y,q,h|a)\\
&+x^{a+b}y^2\HSP(x,y,q,h|b)+x^ay\sum_{j\geq b+1}\HSP(x,y,q,h|bj)\\
&=x^{a+b}y^2+x^{a+b}y^2\sum_{j=1}^{b-1}\HSP(x,y,q,h|j)+(h^{b-a}q-1)x^{a+b}y^2\HSP(x,y,q,h|a)\\
&+x^{a+b}y^2\HSP(x,y,q,h|b)+x^ay\sum_{j\geq b+1}\HSP(x,y,q,h|bj).\\
\end{align*}
By (\ref{eq2}) we get
\begin{align*}
\HSP(x,y,q,h|ab)&=x^{a+b}y^2+x^{a+b}y^2\sum_{j=1}^{b-1}\HSP(x,y,q,h|j)+(qh^{b-a}-1)x^{a+b}y^2\HSP(x,y,q,h|a)\\
&+x^{a+b}y^2\HSP(x,y,q,h|b)+x^ay(\HSP(x,y,q,h|b)
-x^by-x^by\sum_{j=1}^{b}\HSP(x,y,q,h|j))\\
&=(qh^{b-a}-1)x^{a+b}y^2\HSP(x,y,q,h|a)+x^ay\HSP(x,y,q,h|b).
\end{align*}
Summing over all $b\geq a+1$ we have
\begin{align*}
\sum_{j\geq a+1}\HSP(x,y,q,h|aj)&=\sum_{j\geq a+1}(qh^{j-a}-1)x^{a+j}y^2\HSP(x,y,q,h|a)+x^ay\sum_{j\geq a+1}\HSP(x,y,q,h|j)\\
&=x^{2a+1}y^2\HSP(x,y,q,h|a)\left(\frac{qh}{1-hx}-\frac{1}{1-x}\right)+x^ay\sum_{j\geq a+1}\HSP(x,y,q,h|j).
\end{align*}
Thus, by (\ref{eq2}) we have
\begin{align*}
\HSP(x,y,q,h|a)&=x^ay+\sum_{j=1}^{a}\HSP(x,y,q,h|aj)+x^{2a+1}y^2\HSP(x,y,q,h|a)\left(\frac{qh}{1-hx}-\frac{1}{1-x}\right)\\
&+x^ay\sum_{j\geq a+1}\HSP(x,y,q,h|j)\\
&=x^ay+x^ay\sum_{j=1}^{a}\HSP(x,y,q,h|j)+x^{2a+1}y^2\HSP(x,y,q,h|a)\left(\frac{qh}{1-hx}-\frac{1}{1-x}\right)\\
&+x^ay\sum_{j\geq a+1}\HSP(x,y,q,h|j)\\
&=x^ay+x^ay\sum_{j\geq 1}\HSP(x,y,q,h|j)+x^{2a+1}y^2\HSP(x,y,q,h|a)\left(\frac{qh}{1-hx}-\frac{1}{1-x}\right)\\
&=x^ay+x^ay(\HSP(x,y,q,h)-1)+x^{2a+1}y^2\HSP(x,y,q,h|a)\left(\frac{qh}{1-hx}-\frac{1}{1-x}\right).
\end{align*}
Therefore,
\begin{align*}
\HSP(x,y,q,h|a)=\frac{x^ay\HSP(x,y,q,h)}{1-x^{2a+1}y^2\left(\frac{qh}{1-hx}-\frac{1}{1-x}\right)}.
\end{align*}
Hence, by (\ref{eq1}) we have
\begin{align*}
\HSP(x,y,q,h)-1=\HSP(x,y,q,h)\sum_{a\geq 1}\frac{x^ay}{1-x^{2a+1}y^2\left(\frac{qh}{1-hx}-\frac{1}{1-x}\right)}.
\end{align*}
This leads to the required result.
\end{proof}

Note that by substituting $h=1$ in (\ref{eq11}) we obtain the generating function for the number of compositions of $n$ with $k$ parts according to the number of  symmetric peaks
\begin{equation}\label{eq3}
\HSP(x,y,q,1)=\frac{1}{1-\sum_{a\geq1}\frac{x^ay}{1-y^2x^{2a+1}(\frac{q-1}{1-x})}}.
\end{equation}
 which is in accord with the result in \cite{MMR}.

Let $\hsympeak(n)$ denote the sum of heights of symmetric peaks in $C_n$, and let $\hsympeak(n,k)$ denote the the sum of heights of symmetric peaks in $C_{n,k}$. Then $\hsympeak(n)=\sum_{k=0}^{n}\hsympeak(n,k)$.

According to Theorem \ref{Th1}, we can conclude this lemma.
\begin{lemma}\label{lem1}
The generation functions for the sequences $\hsympeak(n)$ and $\hsympeak(n,k)$ are given by
 \begin{equation}\label{eq4}
 \sum_{n\geq 0}\hsympeak(n)x^n=\frac{\partial }{\partial q}\HSP(x,1,1,h)\mid_{h=1}=\frac{x^4}{(1-2x)^2(1-x^3)}.
 \end{equation}
\begin{equation}\label{eq5}
 \sum_{n,k\geq 0}\hsympeak(n,k)x^ny^k=\frac{\partial }{\partial q}\HSP(x,y,1,h)\mid_{h=1}=\frac{y^3x^4}{(1-x^3)(1-x-yx)^2}.
 \end{equation}
\end{lemma}
Using a computer algebra system such as Maple, we determined the coefficients of $x^n$ in (\ref{eq4}) and obtained the following result.
\begin{corollary}\label{cor1}
The sum of the heights of symmetric peaks over all compositions of $n$ is given by
\begin{equation}\label{eq51}
 \hsympeak(n)=\left(\frac{7n-24}{49}\right)2^{n-1}+\frac{(-33-15i\sqrt{3})(-2)^n}{441(1+i\sqrt{3})^{n+1}}+\frac{(-33+15i\sqrt{3})(-2)^n}{441(1-i\sqrt{3})^{n+1}}+\frac{1}{3}.
\end{equation}
\end{corollary}
For instance, when $n=5$, we have three symmetric peaks over the compositions: $131$, $1211$, and $1121$. The sum of the heights of symmetric peaks is equal to $4$ as obtained from equation (\ref{eq51}).
\subsection{Combinatorial derivation of the number of compositions of $n$ with exactly $k$ parts according to sum of symmetric peak heights }
\begin{theorem}
\begin{enumerate}
\item For $n\geq 0$ and $k>3$, we have
$$\hsympeak(n,k)=(k-2)\sum_{m=2}^{n-k+1}\sum_{b=1}^{min\{m-1,t\}}\binom{n-2b-m-1}{k-4}(m-b) $$
where $t=\lfloor \frac{n-m-k+3}{2} \rfloor$.
\item When $k=3$ we have
$$\hsympeak(n,3)=\sum_{m=2}^{ n-2}h_m$$
where $h_m$ is defined as follows:
\begin{equation}
  h_m=\begin{cases}
    \frac{3m-n}{2}, & \text{if $\frac{n-m}{2}$ is an  integer}.\\
    0, & \text{otherwise}.
  \end{cases}
\end{equation}
\end{enumerate}
\end{theorem}

\begin{proof}
\begin{enumerate}
\item Let $\pi_{i-1}\pi_{i}\pi_{i+1}=bmb$ be a symmetric peak in a composion of $n$ with $k\geq 4$ parts, where $ 2\leq i\leq k-1$. According to \cite{MMR}, the total number of symmetric peaks in $C_n$ with $k$ parts is given by
$$(k-2)\sum_{m=2}^{n-k+1}\sum_{b=1}^{\min\{m-1,t\}}\binom{n-2b-m-1}{k-4}.$$
Multiplying the above expression by $m-b$ for each pair of $m$ and $b$ such that  $2\leq m\leq n-k+1$ and $1\leq b\leq \min\{m-1,t\}$  we get the desired result.
\item Let $h_m=m-b$. If $\frac{n-m}{2}$ is  an integer, then $b=\frac{n-m}{2}$. In this case, $h_m=m-b=m-\frac{n-m}{2}$. Otherwise, there is no symmetric peak $bmb$ such that $\frac{n-m}{2}$ is an integer, and therefore, $h_m=0$.
\end{enumerate}
\end{proof}

\section{Depths of symmetric valleys}
\subsection{The generating function for the number of compositions of $n$ according to the sum of symmetric valleys depth}
Let $\symvalley(\pi)$ and $\dsymvalley(\pi)$ denote the number of symmetric valleys and  the sum of the depths of symmetric valleys in a composition in $\pi$, respectively. Denote the generating function for the number of compositions of $n$ with exactly $k$ parts
according to the statistics  $\symvalley$ and $\dsymvalley$ by $\DSV(x,y,p,d)$; that is
$$\DSV(x,y,p,d)=\sum_{n,k\geq 0}\sum_{\pi \in C_{n,k}}x^n y^k p^{\symvalley(\pi)} d^{\dsymvalley(\pi)}.$$

 We derive an explicit formula for the generating function $\DSV(x,y,p,d)$. To do this, we denote the generating function for the number of compositions $\pi=\pi_1\pi_2\dots\pi_k\in C_{n,k}$ such that $\pi_j=a_j$ for all $j=1,2,\dots,s$ and $s<k$ according to the statistics sv and $\dsymvalley$ by
$$\DSV(x,y,p,d|a_1a_2\cdots a_s)=\sum_{n,k\geq 0}\sum_{\pi=a_1a_2\cdots a_s\pi_{s+1}\cdots\pi_k \in C_{n,k}}x^n y^k p^{\symvalley(\pi)} d^{\dsymvalley(\pi)}.$$

\begin{theorem}\label{Th2}
The generating function for the number of compositions of $n$ with exactly $k$ parts
according to the number of symmetric valleys and the sum of depths of symmetric valleys is given by
\begin{equation}\label{eq15}
\DSV(x,y,p,d)=\frac{1}{1-\sum_{a\geq1}\frac{x^ay}{1-y^2x^{a+1}\left(\frac{px^{a-1}-d^{a-1}}{\frac{x}{d}-1}-\frac{x^{a-1}-1}{x-1}\right)}}.
\end{equation}
\end{theorem}
\begin{proof}
By  definition
\begin{equation}\label{eq6}
\DSV(x,y,p,d)=1+\sum_{a\geq1}\DSV(x,y,p,d|a).
\end{equation}
Moreover
\begin{align}
\DSV(x,y,p,d|a)&= x^ay+\sum_{j=1}^{a-1}\DSV(x,y,p,d|aj)+\sum_{j\geq a}\DSV(x,y,p,d|aj) \nonumber \\
&=x^ay+\sum_{j=1}^{a-1}\DSV(x,y,p,d|aj)+x^ay\sum_{j\geq a}\DSV(x,y,p,d|j).\label{eq7}
\end{align}
For all $b\geq a$
\begin{align*}
\DSV(x,y,p,d|ab)=&x^{a+b}y^2+\sum_{j=1}^{b-1}\DSV(x,y,p,d|abj)+\DSV(x,y,p,d|aba)\\
&+\sum_{j\geq b, j\neq a}\DSV(x,y,q,h|abj)\\
&=x^{a+b}y^2+x^{a}y\sum_{j=1}^{b-1}\DSV(x,y,p,d|bj)+x^{a+b}y^2\sum_{j\geq b, j\neq a}\DSV(x,y,p,d|j)\\
&+x^{a+b}y^2pd^{a-b}\DSV(x,y,p,d|a)\\
&=x^{a+b}y^2+x^ay\sum_{j=1}^{b-1}\DSV(x,y,p,d|bj)+x^{a+b}y^2\DSV(x,y,p,d|a)(pd^{a-b}-1)\\
&+x^{a+b}y^2\sum_{j\geq b}\DSV(x,y,p,d|j).
\end{align*}
By (\ref{eq7}) we get
\begin{align*}
\DSV(x,y,p,d|ab)=&x^{a+b}y^2+x^ay(\DSV(x,y,p,d|b)-x^by-x^by\sum_{j \geq b}\DSV(x,y,p,d|j))\\
&+x^{a+b}y^2\DSV(x,y,p,d|a)(pd^{a-b}-1)+x^{a+b}y^2\sum_{j\geq b}\DSV(x,y,p,d|j)\\
&= x^{a+b}y^2\DSV(x,y,p,d|a)(pd^{a-b}-1)+x^ay\DSV(x,y,p,d|b).
\end{align*}
Thus, by (\ref{eq7})  and by summing over all $b\geq a$ we get
\begin{align*}
\DSV(x,y,p,d|a)=&x^ay+\sum_{j=1}^{a-1}x^{a+j}y^2\DSV(x,y,p,d|a)(pd^{a-j}-1)\\
&+x^ay\sum_{j=1}^{a-1}\DSV(x,y,p,d|j)+x^ay\sum_{j \geq a}\DSV(x,y,p,d|j)\\
&=x^ay\DSV(x,y,p,d)+x^ay^2pd^a\DSV(x,y,p,d|a)\left(\frac{\left(\frac{x}{d}\right)^a-\frac{x}{d}}{\frac{x}{d}-1}\right)\\
&-x^ay^2\DSV(x,y,p,d|a)\left(\frac{x^a-x}{x-1}\right).
\end{align*}
Therefore,
\begin{align*}
\DSV(x,y,p,d|a)=&\frac{x^ay}{1-y^2x^{a+1}\left(\frac{px^{a-1}-d^{a-1}}{\frac{x}{d}-1}-\frac{x^{a-1}-1}{x-1}\right)}\DSV(x,y,p,d)
\end{align*}
Hence, by (\ref{eq6}) we get
\begin{align*}
\DSV(x,y,p,d)=\frac{1}{1-\sum_{a\geq 1}\frac{x^ay}{1-y^2x^{a+1}\left(\frac{px^{a-1}-d^{a-1}}{\frac{x}{d}-1}-\frac{x^{a-1}-1}{x-1}\right)}}.
\end{align*}
\end{proof}

Note that by substituting $d=1$ in (\ref{eq15}) we obtain the generating function for the number of compositions of $n$ with $k$ parts according to the number of  symmetric valleys
\begin{equation}\label{eq61}
\DSV(x,y,p,1)=\frac{1}{1-\sum_{a\geq 1}\frac{x^ay}{1-\left(\frac{y^2x^{a+1}(x^{a-1}-1)(p-1)}{x-1}\right)}},
\end{equation}
which is in accord with the result in \cite{MMR} .

Let $\dsymvalley(n)$ denote the sum of depths of symmetric valleys in $C_n$, and let $\dsymvalley(n,k)$ denote the sum of depths of symmetric valleys in $C_{n,k}$. Then $\dsymvalley(n)=\sum_{k=0}^{n}dsv(n,k)$.

Accordung to Theorem \ref{Th2}, we can conclude this lemma.
\begin{lemma}\label{lem2}
The generating functions for the sequences $\dsymvalley(n)$ and $\dsymvalley(n,k)$ are given by
 \begin{equation}\label{eq8}
\sum_{n\geq 0}\dsymvalley(n)x^n= \frac{\partial }{\partial q}\DSV(x,1,1,d)\mid_{d=1}=\frac{x^7+x^5-2x^6}{(1-2x)^2(1-x^3)(1-x^2)^2}.
 \end{equation}
\begin{equation}\label{eq9}
 \sum_{n,k\geq 0}\dsymvalley(n,k)x^ny^k=\frac{\partial }{\partial q}\DSV(x,y,1,d)\mid_{d=1}=\frac{2x^7y^3-x^6y^3-x^4y^3}{(1-x^3)(1-x^2)(1-x-yx)^2}.
 \end{equation}
\end{lemma}
Using a computer algebra system such as Maple, we determined the coefficients of $x^n$ in (\ref{eq8}) and obtained the following result.
\begin{corollary}\label{cor2}
The sum of depths of symmetric valleys over all compositions of $n$ is given by
\begin{align}
\dsymvalley(n)=\left(\frac{-6n+7}{108}\right)(-1)^{n}+\left(\frac{21n-79}{1323}\right)2^n \nonumber \\
+\frac{(-33-15\sqrt{3}i)\left(-2\right)^n}{441(1+i\sqrt{3})^{n+1}}+\frac{(-33+15\sqrt{3}i)\left(-2\right)^n}{441(1-i\sqrt{3})^{n+1}}+\frac{1}{12}.\label{eq91}
 \end{align}
\end{corollary}
For instance, when $n=8$, there are $15$ symmetric valleys over the compositions: $323$, $3131$, $1313$, $2123$, $3212$, $21221$, $21212$, $22121$, $12122$, $12212$,  $212111$, $111212$, $112121$, and $121211$. The sum of depths of symmetric valleys is equal to $17$ as obtained in equation \ref{eq91}.
\subsection{Combinatorial derivation of the number of compositions of $n$ with exactly $k$ parts according to sum of symmetric valley heights }

\begin{theorem}
\begin{enumerate}
\item For $n\geq 0$, and $k>3$ we have
$$\dsymvalley(n,k)=(k-2)\sum_{m=1}^{\lfloor\frac{n-k+1}{3}\rfloor}\sum_{b=m+1}^{\lfloor\frac{n-m-(k-3)}{2}\rfloor}\binom{n-2b-m-1}{k-4}(b-m). $$
\item When $k=3$ we have
$$\dsymvalley(n,3)=\sum_{m=1}^{ \lfloor \frac{n-2}{3}\rfloor}d_m$$
where $d_m$ is defined as follows:
\begin{equation}
  d_m=\begin{cases}
    \frac{n-3m}{2}, & \text{if $\frac{n-m}{2}$ is an  integer}.\\
    0, & \text{otherwise}.
  \end{cases}
\end{equation}
\end{enumerate}
\end{theorem}

\begin{proof}
\begin{enumerate}
\item Let $\pi_{i-1}\pi_{i}\pi_{i+1}=bmb$ represent a symmetric valley within a composition $\pi$ of $n$ with $k\geq 4$ parts, where $ 2\leq i\leq k-1$  and $1\leq m \leq \lfloor\frac{n-k+1}{3}\rfloor$. Notice that the minimum value of $b$ is $m+1$, and $b$ reaches its maximum value when $\pi_j=1$ for all $1\leq j\leq k$ with $j\neq i-1, i,i+1$. Therefore, $m+1\leq b\leq \lfloor\frac{n-m-(k-3)}{2}\rfloor$.

 The number of compositions of the form $\pi$ is $\binom{n-2b-m-1}{k-4}$, which is equivalent to determining the number of solutions to  $\ell_1+\ell_2\dots+\ell_{k-3}= n-2b-m-(k-3)$, where $\ell_1,\ell_2,\ldots\ell_{k-3}$ are  nonnegative integers. The depth of the symmetric valley $bmb$ is $b-m$. Therefore, the total sum of depths of compositions of $n$ with a symmetric valley in the $i$-th position can be expressed as follows:
\begin{equation*}
\sum_{m=1}^{\lfloor\frac{n-k+1}{3}\rfloor}\sum_{b=m+1}^{\lfloor\frac{n-m-(k-3)}{2}\rfloor}\binom{n-2b-m-1}{k-4}(b-m).
\end{equation*}
 There are $k-2$ options for choosing $i$. Multiplying the above expression by $k-2$ yields the desired result.
\item Let  $d_m=b-m$. If  $\frac{n-m}{2}$ is  an integer then $b=\frac{n-m}{2}$. In this case, $d_m=b-m=\frac{n-m}{2}-m$. Otherwise, there is no symmetric valley $bmb$ such that $\frac{n-m}{2}$ is an integer, and therefore, $d_m=0$.
\end{enumerate}
\end{proof}

\section{ A discrete geometrically distributed sample}
\subsection{The expectation }
In this section we explore discrete geometrically distributed samples under the statistics of  $\sympeak, \symvalley, \hsympeak$, and $\dsymvalley$.
\begin{theorem}
\begin{enumerate}
\item The expected value of the statistic $\sympeak$ in a sample of $n$ geometric random variables is given by
$$\frac{p^2q}{1-q^3}(n-2).$$
\item The expected value of the statistic $\symvalley$ in a sample of $n$ geometric random variables is given by
$$p^2\left(\frac{1}{1-q^2}-\frac{1}{1-q^3}\right)(n-2).$$
\item The expected value of the statistic $\hsympeak$ in a sample of $n$ geometric random variables is given by
$$\frac{pq}{1-q^3}(n-2).$$
\item The expected value of the statistic $\dsymvalley$ in a sample of $n$ geometric random variables is given by
$$q\left(\frac{p}{1-q^3}-\frac{1}{(1+q)^2}\right)(n-2).$$
\end{enumerate}
\end{theorem}
\begin{proof}
\begin{enumerate}
\item By mapping $x\mapsto q$, $y\mapsto \frac{pz}{q}$, and $q \mapsto u$ in  equation (\ref{eq3}) we obtain the generating function for the number of discrete geometrically distributed samples of length $n$ according to the statistic $\sympeak$, which is given by

\begin{equation}\label{eq29}
\GSP(z,u)=\frac{1}{1-\sum_{a\geq1}\frac{q^{a-1}pz}{1-pz^2q^{2a-1}(u-1)}}.
\end{equation}
To find  the expected value of the statistic sp in a sample of $n$ geometric random variables, we use the expression $\frac{[z^n]\frac{d}{du}\GSP(z,u)|_{u=1}}{[z^n]\GSP(u,1)}$.
It is obvious that $\GSP(u,1)=\frac{1}{1-z}$, hence $[z^n]\GSP(u,1)=1$. From  (\ref{eq29}) we get
\begin{align*}
\frac{d}{du}\GSP(z,u)|_{u=1}=\frac{p^2z^3q}{(1-q^3)}\sum_{k\geq0}\binom{k+1}{k}z^k.
\end{align*}
Therefore,
\begin{align*}
[z^n]\frac{d}{du}\GSP(z,u)|_{u=1}=[z^{n-3}]\frac{p^2z^3q}{(1-q^3)}\sum_{k\geq0}\binom{k+1}{k}z^k=\frac{p^2q}{1-q^3}(n-2).
\end{align*}

\item Under the mapping $x\mapsto q$, $y\mapsto \frac{pz}{q}$, and $p \mapsto u$, and by substituting $d=1$ in equation (\ref{eq15}), we define $\GSV(z,u)$ as the generating function for the number of discrete geometrically distributed samples of length $n$ according to the statistic $\symvalley$. By finding the coefficients of $z^n$ in $\frac{\partial }{\partial u}\GSV(z,u)\mid_{u=1}$, similarly to above we obtain the required result.

\item Under the  mapping $x\mapsto q$,  $y\mapsto \frac{pz}{q}$, and  $h \mapsto u$, and by substituting $q=1$ in equation (\ref{eq11}), we define $\GHSP(z,u)$ as the generating function for the number of discrete geometrically distributed samples of length $n$ according to the statistic $\hsympeak$.  By finding the coefficients of $z^n$ in $\frac{\partial }{\partial u}\GHSP(z,u)\mid_{u=1}$, similarly to above we obtain the required result.
\item By mapping $x\mapsto q$, $y\mapsto \frac{pz}{q}$, and $d \mapsto u$, and by substituting $p=1$ in equation (\ref{eq15}), we define $\GDSV(z,u)$ as the generating function for the number of discrete geometrically distributed samples of length $n$ according to the statistic $\dsymvalley$. By finding the coefficients of $z^n$ in $\frac{\partial }{\partial u}\GDSV(z,u)\mid_{u=1}$, similarly to above we obtain the required result.
\end{enumerate}
\end{proof}
\subsection{The variance }
\begin{theorem}
\begin{enumerate}
\item The variance of the number of symmetric peaks of a sample of $n$ geometric random variables is given by
$$2(n-4)p^3q^2\left(\frac{p(n-5)}{24(1-q)^3}+\frac{1}{1-q^5}\right)+(n-2)\frac{p^2q}{1-q^3}-(n-2)^2\frac{p^4q^2}{(1-q^3)^2}.$$
\item The variance of the sum of symmetric peak heights of a sample of $n$ geometric random variables is given by
$$(n-5)(n-6)\frac{p^2q^2}{(1-q^3)^2}+2(n-4)\frac{pq^2}{(1-q^5)}+\frac{(n-2)}{(1-q^3)}(2q^2+pq)-\frac{p^2q^2}{(1-q^3)^2}(n-2)^2.$$
\item The variance of the number of symmetric valleys of a sample of $n$ geometric random variables is given by
$$\frac{p^4q^2}{1-q^3}(n-5)(n-6)+\frac{2p^3}{1-q^5}(n-5)+p^2\left(\frac{1}{1-q^2}-\frac{1}{1-q^3}\right)(n-2)-p^4\left(\frac{1}{1-q^2}-\frac{1}{1-q^3}\right)^2(n-2)^2.$$
\end{enumerate}
\end{theorem}

\begin{proof}
\begin{enumerate}
\item To find  the variance of the statistic sp in a sample of $n$ geometric random variables, we use the expression
 $$\frac{[z^n]\frac{d^2}{du^2}\GSP(z,u)|_{u=1}}{[z^n]\GSP(u,1)}+\frac{[z^n]\frac{d}{du}\GSP(z,u)|_{u=1}}{[z^n]\GSP(u,1)}-\left(\frac{[z^n]\frac{d}{du}\GSP(z,u)|_{u=1}}{[z^n]\GSP(u,1)}\right)^2.$$
According to previous results, we have
 $$\frac{[z^n]\frac{d}{du}\GSP(z,u)|_{u=1}}{[z^n]\GSP(u,1)}=\frac{p^2q}{1-q^3}(n-2).$$
In addition,
$$\frac{[z^n]\frac{d^2}{du^2}\GSP(z,u)|_{u=1}}{[z^n]\GSP(u,1)}=2(n-4)p^3q^2\left(\frac{p(n-5)}{4!(1-q)^3}+\frac{1}{1-q^5}\right).$$
\end{enumerate}
Using a similar approach, we achieve $2$ and $3$.
\end{proof}
We invite readers to derive the formula for the variance of the sum of symmetric valley depths for a sample of $n$ geometric random variables.

\end{document}